\documentclass{amsart}

\newtheorem{theorem}{Theorem}
\theoremstyle{plain}

\newtheorem{corollary}{Corollary}

\newtheorem{definition}{Definition}
\newtheorem{example}{Example}

\newtheorem{proposition}{Proposition}

\numberwithin{equation}{section}

\begin{document}
\title[HERMITE-HADAMARD TYPE INEQUALITY]{HERMITE-HADAMARD TYPE INEQUALITY FOR OPERATOR PREINVEX FUNCTIONS }
\author[A. G. Ghazanfari et al.]{A. G. Ghazanfari$^{1,*}$, M. Shakoori$^{2}$, A. Barani$^{3}$, S. S. Dragomir$^{4}$}
\thanks{*Corresponding author.\\
E-mail addresses:\\
ghazanfari.amir@gmail.com(A.G. Ghazanfari),
sever.dragomir@@vu.edu.au(S.S. Dragomir),
 alibarani2000@yahoo.com(A. Barani),
mahmoodshakoori@gmail.com (Mahmood Shakoori)}

\address{}
\email{} \maketitle {\small  \centerline{ \em $^{1,2,3}$Department of
Mathematics, Lorestan University}
\centerline{ \em P. O. Box 465, Khoramabad, Iran}
\centerline{ \em $^{4}$ School of Engineering and Science, Victoria University}
\centerline{ \em PO Box 14428 Melbourne City, MC 8001, Australia.}}

\begin{abstract}
In this paper we establish a Hermite-
Hadamard type inequality for operator preinvex
functions and an estimate of the right hand side of a Hermite-
Hadamard type inequality in which some operator preinvex
functions of selfadjoint operators in Hilbert spaces are involved.
\end{abstract}
\section{introduction}

The following inequality holds for any convex function $f$ defined on $\mathbb{R}$ and $a,b \in \mathbb{R}$, with $a<b$
\begin{equation}\label{1.1}
f\left(\frac{a+b}{2}\right)\leq \frac{1}{b-a}\int_a^b f(x)dx\leq \frac{f(a)+f(b)}{2}
\end{equation}
Both inequalities hold in the reversed direction if $f$ is concave. We note that
Hermite-Hadamard's inequality may be regarded as a refinement of the concept of convexity and it follows easily from
Jensen's inequality.
The classical Hermite-Hadamard inequality provides estimates of the mean value
of a continuous convex function $f : [a, b] \rightarrow \mathbb{R}$.
Dragomir and Agarwal in \cite{dra1} presented some estimates of the right hand side of a Hermite-
Hadamard type inequality in which some convex functions are involved.
The main results of \cite{dra1} are given by the following theorems.
\begin{theorem}\label{t1}
Assume $a, b \in\mathbb{R}$ with $a < b$ and $f : [a, b] \rightarrow \mathbb{R}$ is a differentiable
function on $(a, b)$. If $|f'|$ is convex on $[a, b]$ then the following inequality holds
true
\begin{equation}\label{1.2}
\left|\frac{f(a)+f(b)}{2}-\frac{1}{b-a}\int_a^bf(x)dx\right|\leq\frac{(b-a)(|f'(a)|+|f'(b)|)}{8}.
\end{equation}
\end{theorem}
In recent years several extensions and generalizations have been considered for
classical convexity. A significant generalization of convex functions is that of invex
functions introduced by Hanson in \cite{han}.

Let $X$ be a vector space, $ x, y \in X$, $ x\neq y$. Define the segment
\[[x, y] :=(1- t)x + ty; t \in [0, 1].\]
We consider the function $f : [x, y]\rightarrow \mathbb{R}$ and the associated function
\begin{align*}
&g(x, y) : [0, 1] \rightarrow \mathbb{R},\\
&g(x, y)(t) := f((1 - t)x + ty), t \in [0, 1].
\end{align*}
Note that f is convex on $[x, y]$ if and only if $g(x, y)$ is convex on $[0, 1]$.
For any convex function defined on a segment $[x, y] \in X$, we have the Hermite-
Hadamard integral inequality (see \cite[p.2]{dra2} and \cite[p.2]{dra3})

\begin{equation}\label{1.3}
f\left(\frac{x+y}{2}\right)\leq \int_0^1 f((1-t)x+ty)dt\leq \frac{f(x)+f(y)}{2},
\end{equation}

which can be derived from the classical Hermite-Hadamard inequality (1.1) for the
convex function $g(x, y) : [0; 1] \rightarrow \mathbb{R}$.

Motivated by the above results we investigate in this paper the operator version
of the Hermite-Hadamard inequality for operator preinvex functions and operator convex functions.

In order to do that we need the following preliminary definitions and results.
Let $A$ be a bounded self adjoint linear operator on a complex Hilbert space $(H; \langle .,.\rangle)$.
The Gelfand map establishes a $*$-isometrically isomorphism $\Phi$ between the set
$C (Sp (A))$ of all continuous functions defined on the spectrum of $A$, denoted $Sp (A)$,
and the $C^*$-algebra $C^*(A)$ generated by $A$  and the identity operator $1_H$ on $H$ as
follows (see for instance \cite[p.3]{fur}).
For any $f, g\in C (Sp (A))$ and any $\alpha,\beta\in\mathbb{C}$ we have
\begin{align*}
&(i)~\Phi(\alpha f+\beta g)=\alpha\Phi(f)+\beta\Phi(g);\\
&(ii)~\Phi(fg)=\Phi(f)\Phi(g)~\text{and}~\Phi(f^*)=\Phi(f)^*;\\
&(iii)~\|\Phi(f)\|=\|f\|:=\sup_{t\in Sp(A)} |f(t)|;\\
&(iv)~\Phi(f_0)=1 ~\text{and}~ \Phi(f_1)=A,~\text{ where } f_0(t)=1~ \text{and } f_1(t)=t, \text{for}~ t\in Sp(A).
\end{align*}
With this notation we define
\[ f (A) := \Phi(f) \text{ for all } f \in C (Sp (A))\]
and we call it the continuous functional calculus for a bounded selfadjoint operator $A$.
If $A$ is a bounded selfadjoint operator and $f$ is a real valued continuous function on $Sp (A)$,
then $f (t) \geq 0$ for any $t \in Sp (A)$ implies that $f (A) \geq 0$, i.e., $f (A)$ is a positive
operator on $H$. Moreover, if both $f$ and $g$ are real valued functions on $Sp (A)$ then
the following important property holds:
\begin{align*}
&(P) &&f (t) \geq g (t) \text{ for any } t \in Sp (A) \text{ implies that } f (A)\geq g (A)
\end{align*}
in the operator order in $B(H)$.

A real valued continuous function $f$ on an interval $I$ is said to be operator convex
(operator concave) if
\begin{align*}
&(OC)\quad &&f ((1 -\lambda)A +\lambda B)\leq (\geq) (1-\lambda) f (A) + \lambda f (B)
\end{align*}
in the operator order in $B(H)$, for all $ \lambda\in [0, 1] $ and for every bounded selfadjoint operators $A$ and $B$
in $B(H)$ whose spectra are contained in $I$.

Dragomir in \cite{dra4} has proved a Hermite-Hadamard type inequality for operator convex function as follows:

\begin{theorem}\label{t2}
Let $f : I \rightarrow \mathbb{R}$ be an operator convex function on the interval $I$. Then
for any selfadjoint operators $A$ and $B$ with spectra in $I$ we have the inequality
\begin{multline*}
\left(f\left(\frac{A+B}{2}\right)\leq\right)\frac{1}{2}\left[f\left(\frac{3A+B}{4}\right)+f\left(\frac{A+3B}{4}\right)\right]\\
\leq\int_0^1 f((1-t)A+tB))dt\\
\leq\frac{1}{2}\left[f\left(\frac{A+B}{2}\right)+\frac{f(A)+f(B)}{2}\right]\left(\leq\frac{f(A)+f(B)}{2}\right).
\end{multline*}
\end{theorem}
In this paper we show that Theorem \ref{t3} holds for operator preinvex functions and establish an estimate of the right
hand side of a Hermite-
Hadamard type inequality in which some operator preinvex
functions of selfadjoint operators in Hilbert spaces are involved.

\section{operator preinvex functions}

\begin{definition}\label{d1}
Let $X$ be a real vector space, a set $S\subseteq X$ is said to be invex with respect
to the map $\eta:S\times S\rightarrow X$, if for every $x,y\in S$
and $t\in[0,1]$,
\begin{equation}\label {2.1}
y+t\eta(x,y)\in S.
\end{equation}
\end{definition}
It is obvious that every convex set is invex with respect
to the map $\eta(x,y)=x-y$, but there exist invex sets which are not convex (see \cite{ant}).

Let $S\subseteq X$ be an invex set with respect to $\eta:S\times S\rightarrow X$.
For every $x,y\in S$ the $\eta-$path $P_{xv}$ joining the points $x$ and $v : =x+\eta(y,x)$  is defined as follows
\[
P_{xv} : = \{z: z = x+t\eta(y,x) : t \in [ 0, 1 ] \}.
\]
The mapping $\eta$ is said to be satisfies the condition  $C$ if for every  $x,y\in S$ and $t \in[0,1]$,

\begin{align*}
(C)\quad\quad \eta(y,y+t\eta(x,y))&=-t\eta(x,y),\\
\eta(x,y+t\eta(x,y))&=(1-t)\eta(x,y).
\end{align*}

Note that for every $x,y\in S$ and every $t_{1},t_{2}\in [ 0, 1 ]$ from condition $C$ we have
\begin{equation}\label{2.2}
\eta(y+t_{2}\eta(x,y),y+t_{1}\eta(x,y))=(t_{2}-t_{1})\eta(x,y),
\end{equation}
see \cite{moh, yan} for details.

Let $\mathcal{A}$ be a $C^*$-algebra, denote by $\mathcal{A}_{sa}$ the set of all self adjoint elements in $\mathcal{A}$.

\begin{definition}\label{d2}
Let $S\subseteq B(H)_{sa}$ be an invex set with respect to $\eta:S\times S\rightarrow B(H)_{sa}$. Then, the continuous function $f:\mathbb{R}\rightarrow {\Bbb R}$ is said to be operator preinvex with respect to $\eta$ on $S$, if for every $A,B\in S$ and $t\in[0,1]$,
\begin{equation}\label {2.3}
f(A+t\eta(B,A))\leq (1-t)f(A)+tf(B).
\end{equation}
in the operator order in $B(H)$.
\end{definition}
Every operator convex function is an operator preinvex with respect
to the map $\eta(A,B)=A-B$ but the converse does not holds (see the following example).

Now, we give an example of some operator preinvex functions and invex sets with respect to the maps $\eta$ which satisfy the conditions (C).

\begin{example}\label{e1}\rm{
\begin{enumerate}
\item[(a)] Suppose that $1_H$ is the identity operator on a Hilbert space $H$, and
\begin{align*}
T:&=(-3\times1_H,-1\times1_H)=\{A\in B(H)_{sa}: -3\times1_H< A< -1\times1_H\}\\
U:&=(1_H,4\times1_H)=\{A\in B(H)_{sa}: 1_H< A< 4\times1_H\}\\
S:&=T\cup U\subseteq B(H)_{sa}.
\end{align*}
 Suppose that the function $\eta_1:S\times S\rightarrow B(H)_{sa}$ is defined by
 \[
\eta_1(A,B) =
\begin{cases}
A-B & \text{ \( A,B\in U \),}\\
A-B & \text{ \( A,B\in T \),}\\
1_H-B & \text{ \( A\in T, B\in U\),}\\
-1_H-B & \text{ \( A\in U, B\in T\)}.
\end{cases}
\]
Clearly $\eta_1$ satisfies condition $C$ and $S$ is an invex set with respect to $\eta_1$.
The real function $f(t)=t^2$ is preinvex with respect to $\eta_1$ on $S$.
 But the real function $g(t)=a+bt,\quad a,b\in \mathbb{R}$ is not
preinvex with respect to $\eta_1$ on $S$.
\item[(b)] Suppose that  $V:=(-2\times1_H,0),~ W:=(0,2\times1_H),~ S:=V\cup W\subseteq B(H)_{sa}$ and the
function $\eta_2:S\times S\rightarrow B(H)_{sa}$ is defined by
 \[
\eta_2(A,B) =
\begin{cases}
A-B & \text{ \( A,B\in V \text{ or }A,B\in W\),}\\
0   & \text{ otherwise }.
\end{cases}
\]
Clearly $\eta_{2}$ satisfies condition $C$ and $S$ is an invex set with respect to $\eta_2$.
The constant functions $f(t)=a,~a\in \mathbb{R}$ is only preinvex functions with respect to $\eta_2$ on $S$.

\item[(c)] The function $f(t) = -|t|$ is not a convex function, but it is a preinvex function with
respect to $\eta_3$, where
\[
\eta_3(A,B) =
\begin{cases}
A-B & \text{ \( A,B\geq 0 \text{ or }A,B\leq 0\),}\\
B-A   & \text{ otherwise }.
\end{cases}
\]

\end{enumerate}
 }
\end{example}

\begin{proposition}\label{p1}
Let $S\subseteq B(H)_{sa}$ be an invex set with respect to $\eta:S\times S\rightarrow B(H)_{sa}$ and $f:\mathbb{R}\rightarrow \Bbb R$ be a continuous function. Suppose that $\eta$ satisfies condition $C$ on $S$. Then for every $A,B\in S$ and $V=A+\eta(B,A)$ the function $f$ is operator preinvex with respect to $\eta$ on $\eta-$path $P_{AV}$ if and only if the function $\varphi_{x,A,B}:[0,1]\rightarrow \Bbb R$ defined by
\begin{equation}\label{2.4}
\varphi_{x,A,B}(t):=\langle f(A+t\eta(B,A))x,x\rangle
\end{equation}
is convex on $[0,1]$ for every $x\in H$ with $\|x\|=1$.
\end{proposition}
\begin{proof}
Suppose that $x\in H$ with $\|x\|=1$ and $\varphi_{x,A,B}$ is convex on $[0,1]$ and $C_{1}:=A+t_{1}\eta(B,A)\in P_{AV}, C_{2}:=A+t_{2}\eta(B,A)\in P_{AV}$. Fix $\lambda \in [0,1]$.
By (\ref{2.4}) we have
\begin{equation}\label {2.5}
\begin{aligned}
\langle f(C_{1}+\lambda\eta(C_{2},C_{1}))x,x\rangle&=\langle f(A+((1-\lambda) t_{1}+\lambda t_{2})\eta(B,A))x,x\rangle\\
&=\varphi_{x,A,B}((1-\lambda) t_{1}+\lambda t_{2})\\
&\leq (1-\lambda)\varphi_{x,A,B}(t_{1})+\lambda\varphi_{x,A,B}(t_{2})\\
&=(1-\lambda) \langle f(C_{1})x,x\rangle+\lambda\langle f(C_{2})x,x\rangle.
\end{aligned}
\end{equation}
Hence, $f$ is operator preinvex with respect to $\eta$ on $\eta-$path $P_{AV}$.

Conversely, let  $A,B\in S$ and the function $f$ be operator preinvex with respect to $\eta$ on $\eta-$path $P_{AV}$. Suppose that $t_{1},t_{2}\in [0,1]$. Then, for every $\lambda \in [0,1]$  and $x\in H$ with $\|x\|=1$ we have
\begin{align}\label {2.6}
\varphi_{x,A,B}((1-\lambda) t_{1}+\lambda t_{2})&=\langle f(A+((1-\lambda) t_{1}+\lambda t_{2})\eta(B,A))x,x\rangle\notag\\
&=\langle f\left(A+ t_{1}\eta(B,A)+\lambda\eta( A+ t_{2}\eta(B,A),A+ t_{1}\eta(B,A)\right)x,x\rangle\notag\\
&\leq \lambda \langle f(A+ t_{2}\eta(B,A))x,x\rangle+(1-\lambda)\langle f(A+ t_{1}\eta(B,A))x,x\rangle\\
&=\lambda\varphi_{x,A,B}(t_{2})+(1-\lambda)\varphi_{x,A,B}(t_{1}).\notag
\end{align}
Therefore, $\varphi_{x,A,B}$ is convex on $[0,1]$.
\end{proof}

\begin{theorem}\label{t3}
Let $S\subseteq B(H)_{sa}$ be an invex set with respect to $\eta:S\times S\rightarrow B(H)_{sa}$ and $\eta$ satisfies condition $C$. If for every $A,B\in S$ and $V=A+\eta(B,A)$ the function $f : I \rightarrow \mathbb{R}$ is operator preinvex with respect to $\eta$ on $\eta-$path $P_{AV}$ with spectra of $A$ and spectra of $V$ in the interval $I$. Then we have the inequality
\begin{align}\label{2.7}
f\left(\frac{A+V}{2}\right)&\leq \frac{1}{2}\left[ f\left(\frac{3A+V}{4}\right)+f\left(\frac{A+3V}{4}\right)\right]\notag\\
&\leq \int_0^1 f(A+t\eta(B,A))dt\\
&\leq\frac{1}{2}\left[ f\left(\frac{A+V}{2}\right)+\frac{f(A)+f(V)}{2}\right]\leq \frac{f(A)+f(B)}{2}\notag.
\end{align}
\end{theorem}

\begin{proof}
For $x\in H$ with $\|x\|=1$ and $t\in [0,1]$, we have
\begin{equation}\label{2.8}
\langle (A+t\eta(B,A))x,x\rangle =\langle Ax,x\rangle+t\langle \eta(B,A)x,x\rangle \in I,
\end{equation}
since $\langle Ax,x\rangle\in Sp(A)\subseteq I$ and $\langle Vx,x\rangle\in Sp(V)\subseteq I$.

Continuity of $f$ and (\ref{2.8}) imply that the operator valued integral $\int_0^1 f(A+t\eta(B,A))dt$
exists. Since $\eta$ satisfied condition $C$, therefore for every $t\in[0,1] $ we have
\begin{align}\label{2.9}
A+\frac{1}{2}\eta(B,A)=A+t\eta(B,A)+\frac{1}{2}\eta(A+(1-t)\eta(B,A), A+t\eta(B,A)).
\end{align}
Preinvexity $f$ with respect to $\eta$ implies that
\begin{align}\label{2.10}
f(A+\frac{1}{2}\eta(B,A))&\leq \frac{1}{2}f(A+t\eta(B,A))+\frac{1}{2}f(A+(1-t)\eta(B,A))\notag\\
&\leq \frac{1}{2}[(1-t)f(A)+tf(B)]+\frac{1}{2}[tf(A)+(1-t)f(B)]\\
&\leq \frac{f(A)+f(B)}{2}\notag.
\end{align}
Integrating the inequality (\ref{2.10}) over $t\in [0, 1]$ and taking into account
that
\begin{equation}\label{2.11}
\int_0^1f(A+t\eta(B,A))dt=\int_0^1f(A+(1-t)\eta(B,A))dt
\end{equation}
then we deduce the Hermite-Hadamard inequality for operator preinvex functions
\begin{align*}
f\left(\frac{A+(A+\eta(B,A))}{2}\right)\leq \int_0^1 f(A+t\eta(B,A))dt\leq \frac{f(A)+f(B)}{2}.
\end{align*}
that holds for any selfadjoint operators $A$ and $B$ with the spectra in $I$.
Define the real-valued function $\varphi_{x,A,B}:[0,1]\rightarrow \mathbb{R}$ given by
$\varphi_{x,A,B}(t)=\langle f(A+t\eta(B,A))x,x\rangle$. Since $f$ is operator preinvex,
by the previous proposition \ref {p1}, $\varphi_{x,A,B}$ is a convex function on $[0,1]$.
Utilizing the Hermite-Hadamard inequality for real-valued convex
functions
\[
\varphi\left(\frac{a+b}{2}\right)\leq\frac{1}{b-a}\int_a^b \varphi(s)ds\leq\frac{\varphi(a)+\varphi(b)}{2}
\]
with $a=0,b=\frac{1}{2}$ we have
\[
\left\langle f\left(\frac{3A+V}{4}\right)x,x\right\rangle\leq 2\int_0^\frac{1}{2} \varphi_{x,A,B}(t)dt\leq\left\langle\frac{f(A)+f\left(\frac{A+V}{2}\right)}{2}x,x\right\rangle
\]
and with $a=\frac{1}{2},b=1$ we have
\[
\left\langle f\left(\frac{A+3V}{4}\right)x,x\right\rangle\leq 2\int_\frac{1}{2}^1 \varphi_{x,A,B}(t)dt\leq\left\langle\frac{f(V)+f\left(\frac{A+V}{2}\right)}{2}x,x\right\rangle
\]
which by summation and division by two produces
\begin{align*}
\left\langle\frac{1}{2}\left[ f\left(\frac{3A+V}{4}\right)+f\left(\frac{A+3V}{4}\right)\right]x,x\right\rangle
&\leq\int_0^1 \langle f(A+t\eta(B,A))x,x\rangle dt\\
&\leq\left\langle\frac{1}{2}\left[ f\left(\frac{A+V}{2}\right)+\frac{f(A)+f(V)}{2}\right]x,x\right\rangle
\end{align*}
Finally, from the continuity of the function $f$ we have
\[
\int_0^1 \langle f(A+t\eta(B,A))x,x\rangle dt=\left\langle\int_0^1 f(A+t\eta(B,A))dtx,x\right\rangle,
\]
and the inequality (\ref{2.9}) implies that
\[
f\left(\frac{A+V}{2}\right)\leq \frac{1}{2}\left[ f\left(\frac{3A+V}{4}\right)+f\left(\frac{A+3V}{4}\right)\right]\leq\frac{f(A)+f(B)}{2}.
\]
Hence we deduce the desired result (\ref{2.7}).

\end{proof}

A simple consequence of the above theorem is that the integral is closer
to the left bound than to the right, namely we can state:
\begin{corollary}\label{c1}
With the assumptions in Theorem \ref{t3} we have the inequality
\begin{align*}
0\leq \int_0^1 f(A+t\eta(B,A))dt-f\left(\frac{A+V}{2}\right)\leq \frac{f(A)+f(B)}{2}-\int_0^1 f(A+t\eta(B,A))dt.
\end{align*}
\end{corollary}
\begin{example}
Let $S,~f,~\eta_1$ be as in Example \ref{e1}, then we have
\begin{multline*}
\left(\frac{A+V}{2}\right)^2\leq \frac{1}{2}\left[ \left(\frac{3A+V}{4}\right)^2+\left(\frac{A+3V}{4}\right)^2\right]\\
\leq \int_0^1 (A+t\eta_1(B,A))^2dt\\
\leq \frac{1}{2}\left[ \left(\frac{A+V}{2}\right)^2+\frac{A^2+V^2}{2}\right]\leq \frac{A^2+B^2}{2},
\end{multline*}
for every $A,B\in S$ and $V=A+\eta_1(B,A)$.
\end{example}

The following Theorem is a generalization of Theorem 3.1 in \cite{bar}.
\begin{theorem}\label{t4}
Let the function $f:I\rightarrow \mathbb{R^+}$ is continuous, $S\subseteq B(H)_{sa}$ be an open invex set with respect to $\eta:S\times S\rightarrow B(H)_{sa}$ and $\eta$ satisfies condition $C$. If
for every $A,B\in S$ and $V=A+\eta(B,A)$ the function $f$  is operator preinvex with respect to $\eta$ on $\eta-$path $P_{AV}$ with spectra of $A$ and spectra of $V$ in $I$. Then, for every $a,b\in (0,1)$ with $a<b$ and every $x\in H$ with $\|x\|=1$ the following inequality holds,
\begin{multline}\label{2.12}
\left|\frac{1}{2}\left\langle\int_0^a f(A+s\eta(B,A))ds~x,x\right\rangle +\frac{1}{2}\left\langle\int_0^b f(A+s\eta(B,A))ds~x,x\right\rangle\right. \\
\left. -\frac{1}{b-a}\int_a^b\left\langle\int_0^t f(A+s\eta(B,A))ds~x,x\right\rangle dt\right|\\
\leq \frac{b-a}{8} \{\langle f(A+a\eta(B,A))x,x\rangle+ \langle f(A+b\eta(B,A))x,x\rangle\}.
\end{multline}
Moreover we have
\begin{multline}\label{2.13}
\left\|\frac{1}{2}\int_0^a f(A+s\eta(B,A))ds +\frac{1}{2}\int_0^b f(A+s\eta(B,A))ds\right. \\
\left. -\frac{1}{b-a}\int_a^b\int_0^t f(A+s\eta(B,A))ds dt \right\|\\
\leq \frac{b-a}{8} \| f(A+a\eta(B,A))+  f(A+b\eta(B,A))\|\\
\leq \frac{b-a}{8}[~ \| f(A+a\eta(B,A))\|+\|f(A+b\eta(B,A))\|~].
\end{multline}
\end{theorem}

\begin{proof}
Let $A,B\in S$ and $a,b\in (0,1)$ with $a<b$. For $x\in H$ with $\|x\|=1$ we define the function $\varphi:[0,1]\rightarrow{\Bbb R}^{+}$ by

\[
\varphi(t):=\left\langle \int_0^tf(A+s\eta(B,A))ds~x,x\right\rangle.
\]
Utilizing the continuity of the function $f$, the continuity property of
the inner product and the properties of the integral of operator-valued functions
we have
\[
\left\langle \int_0^t f(A+s\eta(B,A))ds~x,x\right\rangle=\int_0^t\left\langle f(A+s\eta(B,A))~x,x\right\rangle ds.
\]
Since $f(A+s\eta(B,A))\geq 0$, therefore $\varphi(t)\geq 0$ for all $t\in I$.
Obviously for every  $t\in(0,1)$ we have
\[
\varphi^{\prime}(t)=\langle f(A+t\eta(B,A))x,x\rangle\geq0,
\]
hence, $|\varphi^{\prime}(t)|=\varphi^{\prime}(t)$. Since $f$ is operator preinvex with respect to $\eta$ on $\eta-$path $P_{AV}$, by Proposition \ref{p1} the function $\varphi^{\prime}$ is convex. Applying Theorem \ref{t1} to the function $\varphi$ implies that
\[
\left| \frac{\varphi(a)+\varphi(b)}{2}-\frac{1}{b-a}\int_a^b\varphi(s)ds\right|\leq\frac{(b-a)\left(\varphi^{\prime}(a)+\varphi^{\prime}(b)\right)}{8},
\]
and we deduce that (\ref{2.12}) holds.
Taking supremum  over both side of inequality (\ref {2.12}) for all $x$ with $\|x\|=1$,  we deduce that the inequality (\ref{2.13}) holds.
\end{proof}

\section{Application for operator convex functions}
If we consider $\eta(B,A)=B-A$ in Theorem \ref{t3} then $f : I \rightarrow \mathbb{R}$ will be an operator convex function and $V=B$. Hence we can conclude Theorem \ref{t2} as a result of Theorem \ref{t3}.

As an application of Theorem \ref{t4} we state the following Theorem, which is a generalization of Theorem 2.1 in \cite{bar}.
\begin{theorem}\label{t5}
Let $f : I \rightarrow \mathbb{R^+}$ be an operator convex function on the interval $I$. Then for any selfadjoint operators $A$ and $B$ with spectra in $I$ and $a,b\in (0,1)$ with $a<b$ the following inequality holds,
\begin{multline}\label{3.1}
\left|\frac{1}{2}\left\langle\int_0^a f((1-s)A+sB)ds~x,x\right\rangle+\frac{1}{2}\left\langle\int_0^b
 f((1-s)A+sB)ds~x,x\right\rangle\right. \\
\left. -\frac{1}{b-a}\int_a^b\left\langle\int_0^t f((1-s)A+sB)ds~x,x\right\rangle dt\right|\\
\leq\frac{b-a}{8}\left[\left\langle f((1-a)A+a B)x,x\right\rangle+\left\langle f((1-b)A+b B)x,x\right\rangle\right].
\end{multline}

Moreover we have
\begin{multline}\label{3.2}
\left\|\frac{1}{2}\int_0^a f((1-s)A+sB))ds +\frac{1}{2}\int_0^b f((1-s)A+sB))ds\right. \\
\left. -\frac{1}{b-a}\int_a^b\int_0^t f((1-s)A+sB))ds dt \right\|\\
\leq \frac{b-a}{8} \| f((1-a)A+aB))+  f((1-b)A+bB))\|\\
\leq \frac{b-a}{8}[~ \| f((1-a)A+aB))\|+\|f((1-b)A+bB))\|~].
\end{multline}

\end{theorem}

\end{document}